\documentclass[reqno]{amsart}
\usepackage{hyperref}
\usepackage{graphicx}
\usepackage{subfigure}

\begin{document}

\title[stochastic Cucker-Smale system]
{Flocking and non-flocking  behavior in a stochastic Cucker-Smale system}

\author[]{ Ta Viet Ton$^{\dagger}$, Nguyen Thi Hoai Linh$^{\ddagger}$, Atsushi Yagi$^*$$^{\dagger}$}  
\address{$^{\dagger}$ Department of Applied Physics, Graduate School of Engineering, Osaka University, Suita Osaka 565-0871, Japan}
%\email{taviet.ton@ap.eng.osaka-u.ac.jp}
\address{$^{\ddagger}$ Department of Information and Physical Science, Graduate School of Information Science and Technology, Osaka University, Suita Osaka 565-0871, Japan}


%\subjclass[2000]{34C27, 34D05}
\thanks{$^*$This work is supported by Grant-in-Aid for Scientific Research (No. 20340035) of the Japan Society for the Promotion of Science.}

%\keywords{}

\begin{abstract}
We first present a new stochastic version of the Cucker-Smale model of the emergent behavior in flocks in which the mutual communication between individuals is affected by random factor.  Then,  the existence and uniqueness of global solution to this system are verified. We show a result which agrees with  natural fact  that under the effect of large noise, there is no flocking. In contrast, if noise is small, then  flocking may occur. Paper ends with some numerical examples. 
\end{abstract}

\maketitle
\numberwithin{equation}{section}
\theoremstyle{plain}
\newtheorem{theorem}{Theorem}[section]
\newtheorem{lemma}[theorem]{Lemma}
\newtheorem{proposition}[theorem]{Proposition}
\newtheorem{hypothesis}[theorem]{Hypothesis}
\newtheorem{corollary}[theorem]{Corollary}
\theoremstyle{definition}
\newtheorem{definition}[theorem]{Definition}
\newtheorem{remark}[theorem]{Remark}
\allowdisplaybreaks

\section{Introduction}
Flocking is a prevalent  behavior of most population in natural world such as bacteria, birds, fishes. It is also widespread  in  some phenomena in physics, for example  interacting oscillators. Recently, a number of articles proposed mathematical models for flocking behavior, to name a few  \cite{CS,CS1,FGLO,HTa,JLM,ROS,JSh,VCBO}. Vicsek and collaborators \cite{VCBO}  presented a model (for convenience, we call it Vicsek's model) and then studied flocking behavior via computer simulations. Some theoretical results on the convergence of that model can be found in \cite{JLM}. Based on the Vicsek's model, Cucker and Smale introduced a model for an $N$-particle system \cite {CS,CS1}. Then, some 
mathematicians called it the Cucker-Smale system \cite{AhH,CFRT,MT,HL,HL1,JSh}. We will recall  here some features of this system. We consider motion of  $N$ particles in the space $\mathbb R^d\, (d=1,2,3,\ldots)$. The position of the $i$-th particle is denoted by $x_i = x_i(t)\, (i=1,2,\ldots,N)$. Its velocity is denoted by $v_i = v_i(t)\, (i=1,2,\ldots,N)$.  The Cucker-Smale system is as follows 
%%%%%%%%%%%%%%%%%%%%%%%%   E1
\begin{equation}\label{E1}
\begin{cases}
x_i'=v_i ,\hspace{3cm} 1\leq i\leq N,\\
v_i'=\frac{1}{N} \sum_{j=1}^N\psi(||x_j-x_i||)(v_j-v_i),
\end{cases}
\end{equation}
Here, the weights $\psi(||x_j-x_i||)$ quantify the influence between $i$-th and $j$-th particles. This communication rate is a nonincreasing function $\psi\,{:}\,[0, \infty) \rightarrow [0,\infty)$  of the distances between particles.  This function has various forms. In \cite{CS,CS1}, $\psi(s)=\frac{K}{(c+ s^2)^\beta},$ while in \cite{HL, HL1},  $\psi(s)=\frac{K}{(c+ s^2)^\beta}$, $\psi(s)=\frac{K}{s^{2\beta}}$ or   $\psi(s)=\text{\rm constant}$. For such functions, it is shown that when $\beta <\frac{1}{2}$ the convergence of the velocities to a common velocity  is guaranteed, while for $\beta \geq \frac{1}{2}$ this convergence is guaranteed under some condition on the initial positions and velocities of particles. We call them unconditional flocking and conditional flocking, respectively. In the latter case, the result on the non-flocking for two particles on a line is also stated \cite {CS1}. 

We know that  real systems are often exposed to influences that are  incompletely understood. Therefore extending these deterministic models to ones that embrace more complicated variations is needed. A way of doing that is including stochastic influences or noise.
%Since a real particle model is often subject to environmental noises, it is necessary to describe and analyse flocking by  stochastic models.  
 Up to our knowledge, however, there are few such models which are studied theoretically  \cite{CM,HL1,AhH}. In \cite{CM}, Cucker and Mordecki modified the model \eqref{E1} in $\mathbb R^3$ by adding random noise to it
%%%%%%%%%%%%%%%%%%%%%%%%   E1.11
\begin{equation}\label{E1.11}
\begin{cases}
x_i'=v_i,\hspace{3cm} 1\leq i\leq N,\\
v_i'= \sum_{j=1}^N\psi(||x_j-x_i||)(v_j-v_i)+H_i,
\end{cases}
\end{equation}
Here communication rate $\psi$ has the same form as that in \eqref{E1}. $H_i(t)$ is a three-dimensional Gausian centered, stationary stochastic process, that satisfies a $\delta$\,{-}\,dependence condition for some $\delta>0$, i.e., two sets $\{H_i(s)|s\leq t\}$ and $\{H_i(s)| s\geq t+\delta\}$ are independent for each $t$.
  And $H_i(t)$ has $\mathcal C^0$ trajectories and independent coordinates. 
 The authors showed that  a conditional $\nu$\,{-}\,nearly flocking occurs in finite time with a confidence  which is similar to the conditional flocking of the deterministic system \eqref{E1}. 
That is, for $\nu>0$ small enough, there exists a time $T_0$ depending on $\nu$ and initial values $v_i(0)\, ( 1\leq i\leq N)$  such that for every $t\in [0, T_0)$, $\sum_{i,j} ||v_i(t)-v_j(t)||\leq \nu$ with a positive probability. This probability, however, is not one and the interval time for the occurring is not $[0,\infty)$. In other words, it is just nearly flocking, not flocking. In another approach, Ha and collaborators \cite{HL1, AhH} studied the following two Cucker-Smale systems with  the presence of white noise: 
%%%%%%%%%%%%%%%%%%%%%%%%   E1.1
\begin{equation}\label{E1.1}
\begin{cases}
dx_i=v_idt, \hspace{3cm} 1\leq i\leq N,\\
dv_i= \frac{\lambda}{N}\sum_{j=1}^N\psi(||x_j-x_i||)(v_j-v_i)dt+\sqrt{D}dw_i,
\end{cases}
\end{equation}
and
%%%%%%%%%%%%%%%%%%%%%%%%   E1.2
\begin{equation}\label{E1.2}
\begin{cases}
dx_i=v_idt, \hspace{3cm} 1\leq i\leq N,\\
dv_i=  \frac{\lambda}{N}\sum_{j=1}^N\psi(||x_j-x_i||)(v_j-v_i)dt+D(v_i-v_e)dw_i,
\end{cases}
\end{equation}
where $v_e$ is a constant vector. In \cite{HL1}, the result on the flocking behavior of system \eqref{E1.1} is exhibited when the communication rate is a constant. Furthermore, if the communication rate satisfies a lower bound condition, then the relative fluctuations of  velocities around a mean velocity have a uniformly bounded variance in time. In \cite{AhH}, by giving another definition for flocking which is relative to almost surely convergence, the authors showed flocking of system  \eqref{E1.2} which covers the case of the communication weight employed by Cucker-Smale.

We are interested in the communication rate term in system \eqref{E1}. What happens if this mutual communication is affected by random factor, i.e., $\psi(|x_j-x_i|) \leadsto \psi(|x_j-x_i|)$ + white noise? This motivates us to present and  study a Cucker-Smale system under the effect of a common white noise on the mutual communication. We will show that if the noise is large then there is no flocking. Even if the communication rate satisfies the unconditional flocking for \eqref{E1}, this still holds true. This is  different from the deterministic Cucker-Smale system \eqref{E1} but  is adaptive to  real situations. For example, under a strongly random effect of strong winds or water currents, birds or fishes will separate and can not make a flock or school. In contrast, we will show that if the noise is small enough, then flocking occurs. For the convenience of calculation, we will use the Stratonovich stochastic differential equations. Our model to study in this paper has the form
%%%%%%%%%%%%%%%E3
\begin{equation}\label{E3}
\begin{cases}
dx_i=v_i dt,\\
\hspace*{10cm} 1\leq i\leq N.\\
dv_i= \sum_{j=1}^N\psi(||x_j-x_i||)(v_j-v_i)dt +\sigma  \sum_{j=1}^N (v_j-v_i)\circ dw_t,
\end{cases}
\end{equation}
Here $\sigma>0$ is the strength of white noise and $\{w_t, t\geq 0\}$ is  one-dimensional Brownian motion  defined on a  complete probability space with  normal filtration $(\Omega, \mathcal F,\{\mathcal F_t\}_{t\geq 0},\mathbb P)$. We assume that  communication rate function  $\psi\,{:}\,[0, \infty) \rightarrow [0,\infty) $ is locally Lipschitz continuous.

The organization of the paper is as follows. In the next section, we prove the existence and uniqueness of global solution to \eqref{E3}. In Section 3,  non-flocking under the effect of large noise is shown. In contrast, in Section 4, we exhibit flocking under the effect of small noise. Finally, some numerical examples are presented in Section 5.
\section{Existence and uniqueness of global solution}
In this section, we shall prove global existence of solution for the system \eqref{E3}.
%%%%%%%%%%%%%%%%%%%% Th1 Theorem
\begin{theorem} \label{Th1}
For any given initial values $(x_i(0), v_i(0))\in \mathbb R^{2d} \, (1\leq i\leq N),$ system \eqref{E3} has a unique and global solution.
\end{theorem}
\begin{proof}
Since the functions on the right side of \eqref{E3} are locally Lipschitz continuous on $\mathbb R^{2d},$  there is a unique solution $(x_i(t),v_i(t))\, (1\leq i\leq N)$ defined on an interval $[0,\tau),$ where $\tau\leq \infty$  and if $\tau<\infty$ it is an explosion time \cite{A,F}, i.e.,
$$\tau=\sup\{t\,{:}\,\sup_{s\in [0,t), 1\leq i\leq  N} [||x_i(s)||+||v_i(s)||]<\infty \}.$$
Put 
\begin{align*}
||x||=\sqrt{\sum_{i=1}^N ||x_i||^2}, \quad ||v||=\sqrt{\sum_{i=1}^N ||v_i||^2},
 \quad  \bar x=\frac{1}{N}\sum_{i=1}^N x_i,  \quad  \bar v=\frac{1}{N}\sum_{i=1}^N v_i.
\end{align*}
It follows from system \eqref{E3} that
\begin{equation*}
\begin{cases}
d\bar x=\bar v dt,\\
d\bar v={\bf 0}.
\end{cases}
\end{equation*}
Then almost surely $\bar x(t)=\bar x(0)+\bar v(0) t $ and $\bar v(t)=\bar v(0)$ for every $t\in [0, \tau).$ Without loss of generality, we may assume that $\bar x(0)=\bar v(0)={\bf 0}$. Then,
%%%%%%%%%%%%%%%%%%E6
\begin{equation}\label{E6}
\sum_{i=1}^N x_i(t)=\sum_{i=1}^N v_i(t)={\bf 0} \hspace {2cm}\text{ a.s.}
\end{equation}
and
%%%%%%%%%%%%%%%%%%E7
\begin{equation}\label{E7}
\sum_{i,j=1}^N ||v_i(t)-v_j(t)||^2=2N||v(t)||^2.
\end{equation}
It follows from the second equation of  \eqref{E3} and from the chain rule of Stratonovich stochastic differential equation that
%%%%%%%%%%%%%%% E8
\begin{equation}\label{E8}
\begin{aligned}
d||v||^2=&\sum_{i=1}^Nd||v_i||^2\\
=&2 \sum_{i=1}^N <v_i,dv_i>\\
=&2 \sum_{i,j=1}^N\psi(||x_j-x_i||)<v_i, v_j-v_i>dt +2 \sigma \sum_{i,j=1}^N <v_i, v_j-v_i> \circ dw_t.
\end{aligned}
\end{equation}
We have
\begin{align*}
  \sum_{i,j=1}^N <v_i, v_j-v_i> &= \sum_{i, j=1}^N <v_i-v_j, v_j-v_i>+ \sum_{i, j=1}^N <v_j, v_j-v_i>\\
&=-\sum_{i, j=1}^N||v_i-v_j||^2- \sum_{i, k=1}^N <v_k, v_i-v_k>,
\end{align*}
which induces
%%%%%%%%%%% E9
\begin{align} \label{E9}
2\sum_{i,j=1}^N <v_i, v_j-v_i>&=-\sum_{i, j=1}^N||v_i-v_j||^2\notag\\
&=-2N ||v||^2 \quad  (\text{ see } \eqref{E7}).
\end{align}
Furthermore, it follows from
\begin{align*}
\sum_{i,j=1}^N\psi(||x_j-x_i||)<v_i, v_j-v_i> =&-\sum_{i,j=1}^N\psi(||x_j-x_i||)||v_i-v_j||^2\\
&+\sum_{i,j=1}^N\psi(||x_j-x_i||)<v_j, v_j-v_i>
\end{align*}
that
%%%%%%%%%%%% E10
\begin{align} \label{E10}
2\sum_{i,j=1}^N\psi(||x_j-x_i||)<v_i, v_j-v_i> =-\sum_{i,j=1}^N\psi(||x_j-x_i||)||v_i-v_j||^2.
\end{align}
Thus, by \eqref{E8}-\eqref{E10}, we obtain
\begin{equation*}
\begin{aligned}
d||v||^2
=- \sum_{i,j=1}^N\psi(||x_j-x_i||)||v_i-v_j||^2dt - 2N\sigma ||v||^2\circ dw_t.
\end{aligned}
\end{equation*}
Or, equivalently, in the It\^o form:
%%%%%%%%%%%%%%% E11
\begin{equation}\label{E11}
\begin{aligned}
d||v||^2
=\left [- \sum_{i,j=1}^N\psi(||x_j-x_i||)||v_i-v_j||^2+2N^2\sigma^2 ||v||^2\right]dt - 2N\sigma ||v||^2 dw_t.
\end{aligned}
\end{equation}
Hence, by using the comparison theorem  \cite{WI}, it follows from \eqref{E11} that for every $t\geq 0$, $||v(t)||^2\leq V(t) $ a.s., where $V(t)$ satisfies the following equation
\begin{equation*}
\begin{cases}
dV=2N^2\sigma^2 Vdt - 2N\sigma V dw_t,\\
V(0)=||v(0)||^2.
\end{cases}
\end{equation*}
This linear equation  has a unique global solution $V(t)=V(0)e^{-2N\sigma w_t }.$ Thus for every $t\in [0,\tau)$
%%%%%%%%%%%%%%% E12
\begin{equation}\label{E12}
 ||v(t)||\leq ||v(0)||e^{-N\sigma w_t }\hspace{2cm} \text{a.s. }
\end{equation}
Then, from the first equation of \eqref{E3} we have almost surely
\begin{equation*}
\begin{aligned}
d||x||^2&=\sum_{i=1}^Nd||x_i||^2=2\sum_{i=1}^N <x_i, v_i> dt\\
&\leq 2 ||x|| \, ||v|| dt\leq 2||v(0)||e^{-N\sigma w_t } ||x||dt.
\end{aligned}
\end{equation*}
By the comparison theorem, we obtain $||x(t)||^2 \leq X(t)$ for all $t\geq 0$, where $X(t)$ satisfies the following equation
\begin{equation*}
\begin{cases}
dX=2||v(0)||e^{-N\sigma w_t } \sqrt{X}dt,\\
X(0)=||x(0)||^2>0.
\end{cases}
\end{equation*}
Since $X(t)=[||x(0)|| +||v(0)||\int_0^te^{-N\sigma w_s }ds]^2,$ then for every  $t\in [0,\tau)$ 
%%%%%%%%%%%%%%% E13
\begin{equation}\label{E13}
 ||x(t)||\leq ||x(0)|| +||v(0)||\int_0^te^{-N\sigma w_s }ds \hspace {2cm}\text{ a.s.}
\end{equation}
From \eqref{E12},  \eqref{E13} and the definition of $\tau$, we see that $\tau=\infty$ a.s. It means that the solution to  \eqref{E3} is unique and global.
\end{proof}
\section{Non-flocking under large noise }
In this section, we will show that under the effect of large noise on particles, i.e., $\sigma$ is large, then the system \eqref{E3} does not flock. First, we give a definition for flocking. Then non-flocking theorem is shown.
%%%%%%%%%%%%%%%%%% def1
\begin{definition} \label{def1}
The state of particles $(x_i(t), v_i(t)) \, (1\leq i\leq N)$  in system \eqref{E3} has  a {\it time-asymptotic flocking} if, for $1\leq i, j\leq N$, the velocity alignment and group forming in the following senses, respectively,  are satisfied
\begin{itemize}
  \item [1.] $ \lim_{t\to \infty}\mathbb E ||v_i-v_j||^2=0.$
              %$$\lim_{t\to \infty}||v_i-v_j||=0 \quad \text{ a.s. or } \lim_{t\to \infty}\mathbb E ||v_i-v_j||^2=0.$$  
\item [2.]  $\sup_{0\leq t<\infty} \mathbb E ||x_i-x_j||<\infty.$
\end{itemize}
\end{definition}
Put $\alpha=\sup_{s\geq 0} \psi(s). $ Throughout this section, we assume that the communication rate satisfies an upper bound condition, i.e., $\alpha<\infty$. Note that  the communication rate in Cucker-Smale system \cite{CS,CS1} satisfies this condition. 
%%%%%%%%%%%%%%%%%%%% thm2
\begin{theorem}[Non-flocking theorem]\label{thm2}
If $\sigma > \sqrt{\frac{\alpha}{N}}$ then the particles do not flock.
\end{theorem} 
\begin{proof}
As we see in the proof of Theorem \ref{Th1} that without loss of generality, we can assume that $\bar x(0)=\bar v(0)={\bf 0}.$ Then  \eqref{E6} and \eqref{E7} hold true. It follows from \eqref{E7} and \eqref{E11} that
\begin{equation*}
\begin{aligned}
d||v||^2
=& \sum_{i,j=1}^N[N\sigma^2 -\psi(||x_j-x_i||)]||v_i-v_j||^2dt - 2N\sigma ||v||^2 dw_t\\
\geq &\sum_{i,j=1}^N(N\sigma^2-\alpha)||v_i-v_j||^2dt - 2N\sigma ||v||^2 dw_t\\
=&2N(N\sigma^2-\alpha)||v||^2 dt - 2N\sigma ||v||^2 dw_t.
\end{aligned}
\end{equation*}
By using the comparison theorem, for every $t\geq 0$ we have 
$||v(t)||^2\geq V_1(t) $ a.s., where $V_1(t)$ is a unique solution of the following equation
\begin{equation*}
\begin{cases}
dV_1=2N (N\sigma^2-\alpha)V_1 dt - 2N\sigma V_1 dw_t,\\
V_1(0)=||v(0)||^2>0.
\end{cases}
\end{equation*}
Since $d\mathbb E V_1=2N(N\sigma^2-\alpha)\mathbb E V_1 dt,$ we have $\mathbb E V_1(t)=||v(0)||^2 e^{2N(N\sigma^2-\alpha)t}\rightarrow \infty$ as $t\rightarrow \infty$. Therefore, $\mathbb E ||v(t)||^2\rightarrow \infty$ as $t\rightarrow \infty$. It means that the particles do not  flock.
\end{proof}
\section{Flocking under small noise }
In this section, we consider the system \eqref{E3} under the effect of small noise, i.e., $\sigma$ is small. We will show that flocking takes place under a lower bound condition of the communication rate.
%%%%%%%%%%%%%%%%%%%%Thm3
\begin{theorem}[Flocking theorem]\label{Thm3}
Assume that there exists $ \psi^*>0$ such that
$$\inf_{s\geq 0} \psi (s)\geq \psi^*.$$
If $\sigma<\sqrt{\frac{ \psi^*}{N}}$ then particles flock under any initial values $(x_i(0), v_i(0)) \in \mathbb R^{2d}\, (1\leq i\leq N)$.
\end{theorem}
\begin{proof}
Similar to  Theorem \ref{thm2}, we can assume that $\bar x(0)=\bar v(0)={\bf 0}.$ Then   \eqref{E6} and \eqref{E7} hold true. By \eqref{E7} and \eqref{E11}, we have
\begin{equation*}
\begin{aligned}
d||v||^2
\leq &\left[-\sum_{i, j}^N\psi^*||v_i-v_j||^2+2N^2\sigma^2 ||v||^2\right ]dt - 2N\sigma ||v||^2 dw_t\\
=&\big (N\sigma^2-\psi^*\big )||v||^2dt - 2N\sigma ||v||^2 dw_t.
\end{aligned}
\end{equation*}
By using the comparison theorem, it follows that
$||v(t)||^2\leq V_2(t)$ a.s., where $V_2(t)$ satisfies the following equation
%%%%%%%%%%%%%%% E16
\begin{equation}\label{E16}
\begin{cases}
dV_2=\big (N\sigma^2-\psi^*\big )V_2dt - 2N\sigma V_2 dw_t,\\
V_2(0)=||v(0)||^2.
\end{cases}
\end{equation}
Since $d\mathbb E V_2=\big (N\sigma^2-\psi^*\big )\mathbb E V_2dt$ and $\sigma<\sqrt{\frac{ \psi^*}{N}},$ we have
$$\mathbb E||v(t)||^2 \leq E V_2(t)=||v(0)||^2 e^{ (N\sigma^2-\psi^* )t} \rightarrow 0 \text{   as } t\rightarrow \infty.$$
On the other hand, the linear equation \eqref{E16} has an explicit solution
$$V_2(t)=||v(0)||^2 e^{[(N-2N^2)\sigma^2-\psi^* ]t-2N\sigma w_t}.$$
Thus,
%%%%%%%%%%%%%%%%E16.1
\begin{equation} \label{E16.1}
||v(t)||\leq ||v(0)||e^{\frac{1}{2}\big[(N-2N^2)\sigma^2-\psi^*\big]t-N\sigma w_t}. 
\end{equation}
In the prove of Theorem \ref{Th1}, we showed that  $d||x||^2\leq 2 ||x|| ||v|| dt$. Then by using the comparison theorem and  \eqref{E16.1}, it is easy to obtain that
$$||x(t)||\leq ||x(0)||+ ||v(0)|| \int_0^t e^{\frac{1}{2}[(N-2N^2)\sigma^2-\psi^* ]s-N\sigma w_s}ds.$$
Thus,
\begin{equation*}
\begin{aligned}
\mathbb E ||x(t)|| \leq &||x(0)||+||v(0)||\int_0^t e^{\frac{1}{2}[(N-2N^2)\sigma^2-\psi^*]s}\mathbb Ee^{-N\sigma w_s}ds\\
=&||x(0)||+||v(0)||\int_0^t e^{\frac{1}{2}[(N-2N^2)\sigma^2-\psi^*]s}e^{\frac{1}{2}N^2\sigma^2 s}ds\\
=&||x(0)||+||v(0)||\int_0^t e^{\frac{1}{2}[(N-N^2)\sigma^2-\psi^*]s}ds\\
=&||x(0)||+||v(0)||\frac{2}{(N-N^2)\sigma^2-\psi^*}\left[e^{\frac{1}{2}[(N-N^2)\sigma^2-\psi^*]t}-1\right ].
\end{aligned}
\end{equation*}
Therefore, $\sup_{t\geq 0} \mathbb E ||x(t)||<\infty.$ From the above results, we conclude that the particles flock.
\end{proof}
\begin{remark}
The lower bound condition has been used for the system \eqref{E1.1} by Ha et al.    \cite{HL1}. The authors showed that when this condition holds true, the term $\mathbb \sum_{i,j}E ||v_i-v_j||^2$ is uniformly bounded in $t.$ Furthermore, when  $\psi=1$,  flocking of particles takes place. However, when $\psi$ satisfies the lower bound condition but is not a constant, the flocking result has not yet obtained in \cite{HL1}.
\end{remark}
%%%%%%%%%%%%%%%%%%%% remark5
\begin{remark} \label{remark5}
The above flocking result is relative to  the expectation (average)  of  solution. We can obtain a result on flocking of almost surely convergence that is valid for individual trajectory. To see that, firstly, we change the first condition in Definition \ref{def1} by $\lim_{t\to\infty} ||v_i-v_j||=0$ a.s. Now, it follows from \eqref{E16.1} that 
%%%%%%%%%%%%%%%%%%%%%%%%%%%%%  E19
\begin{equation}\label{E19}
||v(t)||\leq ||v(0)||e^{t\big[\frac{1}{2}(N-2N^2)\sigma^2-\psi^*-N\sigma \frac{w_t}{t}\big]}.
\end{equation}
By the strong law of large numbers for Brownian motion \cite{IS},  $\frac{w_t}{t}\rightarrow 0$ as $t\rightarrow \infty$ a.s. Thus $\lim_{t\to  \infty} ||v(t)|| =0.$ It then follows from \eqref{E7} that for $1\leq i\leq N$, 
%%%%%%%%%%%%%%%%%%%%%% E20
\begin{equation} \label{E20}
\lim_{t\to\infty} ||v_i-v_j||=0 \hspace {2cm} \text{ a.s.}
\end{equation}
 It means that  flocking occurs. Furthermore, it is easy to see from \eqref{E19} that  \eqref{E20} still holds true for the case $\sigma>0, \psi^*=0$.  This means that flocking occurs without  the assumptions on the positivity of $\psi^*$ and upper boundedness of $\sigma$. Consequently,  unconditional flocking  occurs not only for the communication rate $\psi$ in \cite{HL1} but for those in \cite{CS,CS1,CM}.
%Noting that the communication rate function in \cite{CS} has $\psi^*=0.$  
%This is because the communication has a very strong long-range interaction ($\inf_{s\geq 0} \psi (s)\geq \psi^*>0$).
\end{remark}
\section{Numerical examples}
In this section, we present some results of numerical simulations based on Euler's method for the system \eqref{E3}. First, we give examples for non-flocking, second, examples for flocking.
\subsection{Non-flocking} Let us first observe an example that shows that the system 
\eqref{E3} does not flock.

 Set $\alpha=1, \sigma=0.3$  and $\psi(s)=\frac{\alpha}{(1+s^2)^{0.25}}$. We compute $100$ trajectories of solution of  \eqref{E3} in $\mathbb R^2$ with $N=50$ and  initial values $(x_i(0),v_i(0)) $ are randomly generated in $[0,0.1]^4$. Figure 1 illustrates behavior of function $f(t)=\sum_{i<j} \mathbb E ||v_i(t)-v_j(t)||^2$ up to time $T=0.2.$ We see that values of function $f(t)$ are very large even at small time $t$. 
%With the same form of communication rate $\phi$, but $\sigma=0.1 $ and the initial value $(x(0),v(0))$ is randomly generated in $[0,5]$, a conditional flocking (i.e. depends on the initial values) for  \eqref{E3} seem occurs (see Figure 2). We predict that for such communication rate, a flocking takes place when initial values $(x(0,v(0)) $ belong to a certain domain in $\mathbb R^{2d}$ and noise is small.
\begin{figure}[h]
\includegraphics[scale=0.5]{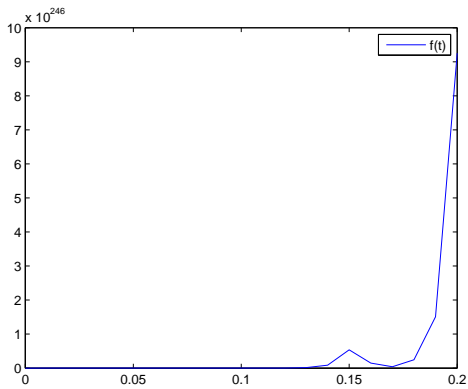}
 \caption{ Non-flocking in $2$-dim. space}
\end{figure}
\begin{figure}[h]
\includegraphics[scale=0.65]{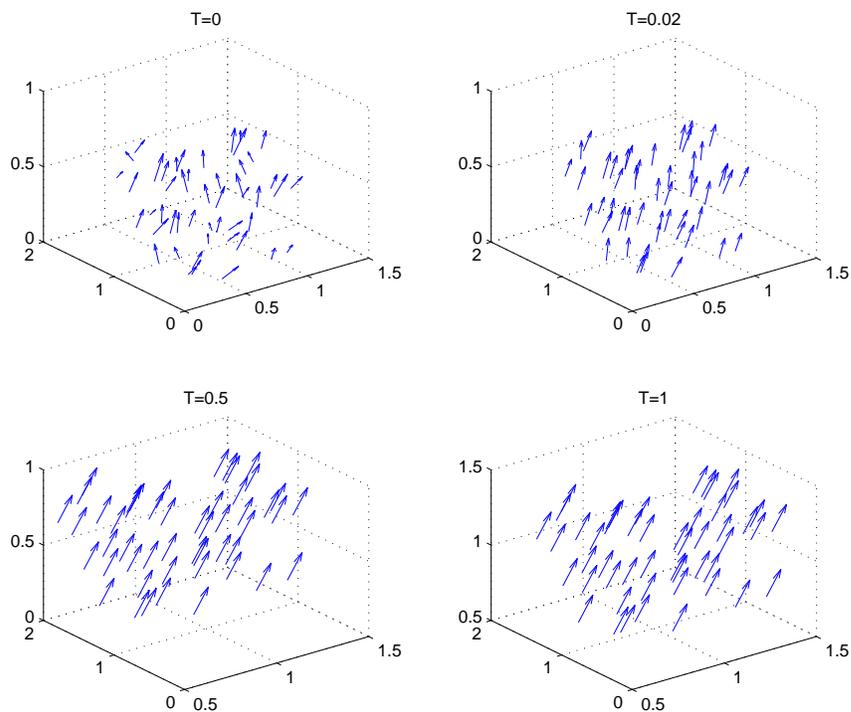}
 \caption{ Flocking in $3$-dim. space}
\end{figure}
\subsection{Flocking} Let us next observe another example showing flocking. Set $\alpha=1, \sigma=0.05$  and $\psi(s)=1.$
 Initial values $(x_i(0),v_i(0))$  are generated randomly in $[0,1]^6$. Figure 2 shows a flock of $50$ particles at $T=0, 0.02, 0.5, 1$. Each vector shows position and direction of motion of each particle.  The lengths of vectors  represent the magnitudes of velocity vectors of particles.
\section{Conclusion}
This paper presented a new stochastic version of the Cucker-Smale model in which the mutual communication between individuals is affected by the common random factor. This model can find its applications  in real world because  almost every real phenomena are  subject to environmental noises.  Under the upper bound condition of communication rate, we show that if the noise is large then the particles can not flock. This is consistent with some situations in nature, for example, under a strongly random effect, there is no flock. In contrast, under the lower bound condition of communication rate and small noise, flocking occurs. The paper ends with some numerical examples of both flocking and non-flocking.

Using the sense of flocking stated in Remark \ref{remark5},   unconditional flocking takes place. However,  a remaining interesting problem for the stochastic Cucker-Smale model is to obtain  flocking results in the sense of Definition \ref{def1}. By numerical simulations, we predict that flocking in this case does not only depend on small noise but on some certain domain of initial values $(x_i(0),v_i(0)).$ This problem is left for our future research.

\end{document}